\documentclass{llncs}
\usepackage{amsfonts,amsmath,latexsym}
\usepackage{enumerate,enumitem}
\usepackage{mathtools,marvosym}

\usepackage{graphicx}
\DeclareGraphicsExtensions{.pdf}
\graphicspath{ {figs/} }

\let\doendproof\endproof
\renewcommand\endproof{~\hfill\qed\doendproof}

\allowdisplaybreaks

\newcommand{\Dsep}{\ensuremath{D_{\operatorname{sep}}}}

\newcommand{\Dlocstar}{\ensuremath{D_{\operatorname{loc-star}}}}

\newtheorem{oldtheorem}{Theorem}

\spnewtheorem*{proofidea}{Proof idea}{\itshape}{}

\title{The number of crossings in multigraphs with no empty lens%
\thanks{A preliminary version~\cite{GD2018version} appeared in the proceedings of the 26th International Symposium on Graph Drawing and Network Visualization, GD 2018.}
}

\author{
Michael Kaufmann\inst{1}
\and
J\'anos Pach\thanks{Supported by NKFIH grants KKP-133864, K-131529, Austrian Science Fund Z 342-N31, Ministry of Education and Science of the Russian Federation MegaGrant No. 075-15-2019-1926, ERC Advanced Grant 882971 ``GeoScape.''}\inst{2}
\and
G\'eza T\'oth\thanks{Supported by National Research, Development and Innovation Office, NKFIH, KKP-133864, K-131529 and ERC Advanced Grant ``GeoScape''882971.}\inst{3}
\and
Torsten Ueckerdt$^{\text{\Letter}}$\inst{4}
}

\institute{
Wilhelm-Schickard-Institut f\"ur Informatik, Universit\"at T\"ubingen, Germany \email{mk@informatik.uni-tuebingen.de} \and
R\'enyi Institute, Budapest, Hungary and MIPT, Moscow, Russian Federation \email{pach@cims.nyu.edu} \and
R\'enyi Institute, Budapest, Hungary \email{toth.geza@renyi.mta.hu} \and
Karlsruhe Institute of Technology (KIT), Institute of Theoretical Informatics, Germany \email{torsten.ueckerdt@kit.edu}
}

\begin{document}

\maketitle

\begin{abstract}
 Let $G$ be a multigraph with $n$ vertices and $e>4n$ edges, drawn in the plane such that any two parallel edges form a simple closed curve with at least one vertex in its interior and at least one vertex in its exterior.
 Pach and T\'oth (A Crossing Lemma for Multigraphs, \textit{SoCG 2018}) 
 extended the Crossing Lemma of Ajtai \textit{et al.} (Crossing-free subgraphs, \textit{North-Holland Mathematics Studies}, 1982)
 and Leighton (Complexity issues in VLSI, \textit{Foundations of computing series}, 1983)
 by showing that if no two adjacent edges cross and every pair of nonadjacent edges cross at most once, then the number of edge crossings in $G$ is at least $\alpha e^3/n^2$, for a suitable constant $\alpha>0$.
 The situation turns out to be quite different if nonparallel edges are allowed to cross any number of times.
 It is proved that in this case the number of crossings in $G$ is at least $\alpha e^{2.5}/n^{1.5}$.
 The order of magnitude of this bound cannot be improved.
\end{abstract}

\section{Introduction}
In this paper, multigraphs may have parallel edges but no loops.
A topological graph (or multigraph) is a graph (multigraph) $G$ drawn in the plane with the property that every vertex is represented by a point and every edge $uv$ is represented by a curve (continuous arc) connecting the two points corresponding to the vertices $u$ and $v$.
We assume, for simplicity, that the points and curves are in ``general position'', that is, (a) no vertex is an interior point of any edge; (b) any pair of edges intersect in at most finitely many points; (c) if two edges share an interior point, then they properly cross at this point; and (d) no 3 edges cross at the same point.
Throughout this paper, every multigraph $G$ is a topological multigraph, that is, $G$ is considered with a fixed drawing that is given from the context.
In notation and terminology, we then do not distinguish between the vertices (edges) and the points (curves) representing them.
The number of crossing points in the considered drawing of $G$ is called its {\em crossing number}, denoted by ${\rm cr}(G)$.
(I.e., ${\rm cr}(G)$ is defined for topological multigraphs rather than abstract multigraphs.)

The classic ``crossing lemma'' of Ajtai, Chv\'atal, Newborn, Szemer\'edi~\cite{ACNS82} and Leighton~\cite{L83} gives an asymptotically best-possible lower bound on the crossing number in any $n$-vertex $e$-edge topological graph without loops or parallel edges, provided $e > 4n$.

\begin{oldtheorem}[Crossing Lemma, Ajtai \textit{et al.}~\cite{ACNS82} and Leighton~\cite{L83}]
 There is an absolute constant $\alpha >0$, such that for any $n$-vertex $e$-edge topological graph $G$ we have
 \[
  {\rm cr}(G) \geq \alpha \frac{e^3}{n^2}, \qquad \text{provided } e > 4n.
 \]
\end{oldtheorem}

In general, the Crossing Lemma does not hold for topological multigraphs with parallel edges, as for every $n$ and $e$ there are $n$-vertex $e$-edge topological multigraphs $G$ with ${\rm cr}(G) = 0$.
Sz\'ekely proved the following variant for multigraphs by restricting the edge multiplicity, that is the maximum number of pairwise parallel edges, in $G$ to be at most $m$.
In fact, the statement holds with the same constant $\alpha$ as the original Crossing Lemma~\cite{PT97}.

\begin{oldtheorem}[Sz\'ekely~\cite{Sz97}]\label{thm:multiplicity-crossing-lemma}
 There is an absolute constant $\alpha > 0$ such that for any $m \geq 1$ and any $n$-vertex $e$-edge topological multigraph $G$ with edge multiplicity at most~$m$ we have
 \[
  {\rm cr}(G) \geq \alpha \frac{e^3}{mn^2}, \qquad \text{provided } e\geq 5mn.
 \]
\end{oldtheorem} 

Recently, Pach and T\'oth extended the Crossing Lemma to so-called branching multigraphs~\cite{PT20}, and together with Tardos to so-called non-homotopic multigraphs~\cite{PTT20}.
We say that a topological multigraph is
\begin{itemize}
 \item {\em separated} if any pair of parallel edges form a simple closed curve with at least one vertex in its interior and at least one vertex in its exterior,
 \item {\em single-crossing} if any pair of edges cross at most once (that is, edges sharing $k$ endpoints, $k \in \{0,1,2\}$, may have at most $k+1$ points in common), 
 \item {\em locally starlike} if no two adjacent edges cross (that is, edges sharing $k$ endpoints, $k \in \{1,2\}$, may not cross), and 
 \item {\em non-homotopic} if no two parallel edges can be continuously transformed into each other without passing through a vertex.
\end{itemize}
A topological multigraph is {\em branching} if it is separated, single-crossing and locally starlike.
Thus every branching drawing is separated, and every separated drawing is non-homotopic.
However, the converse is not true.
The edge multiplicity of a branching multigraph may be as high as $n-2$, while a non-homotopic multigraph with two vertices can already have arbitrarily many edges.

\begin{figure}
 \centering
 \includegraphics{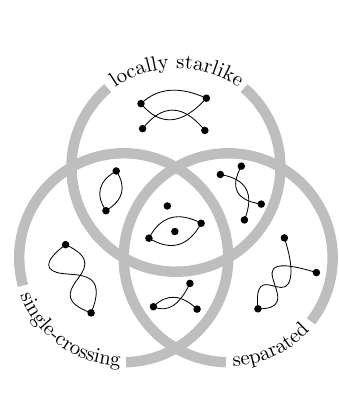}
 \caption{Illustrating some drawing styles of topological multigraphs. A branching drawing is separated, single-crossing and locally starlike.}
 \label{fig:drawing-styles}
\end{figure}

\begin{oldtheorem}[Pach and T\'oth~\cite{PT20}]\label{thm:CL-branching}
 There is an absolute constant $\alpha > 0$ such that for any $n$-vertex $e$-edge branching multigraph $G$ we have
 \[
   {\rm cr}(G) \geq \alpha \frac{e^3}{n^2}, \qquad \text{provided } e > 4n.
 \]
\end{oldtheorem}

\begin{oldtheorem}[Pach, Tardos, and T\'oth~\cite{PTT20}]\label{thm:CL-non-homotopic}
 There is an absolute constant $\alpha > 0$ such that for any $n$-vertex $e$-edge non-homotopic multigraph $G$ we have 
 \[
  {\rm cr}(G) \geq \alpha \frac{e^2}{n}, \qquad \text{provided } e > 4n.
 \]
\end{oldtheorem}

Let us also mention that Felsner \textit{et al.}~\cite{FHKP20} recently considered locally starlike drawings of the complete graph on $n$ vertices in which no face of the arrangement is bounded by a $2$-cycle.
They showed that any such drawing contains at most $n!$ crossings.

In this paper we generalize Theorem~\ref{thm:CL-branching} by showing that the Crossing Lemma holds for all topological multigraphs that are separated and locally starlike, but not necessarily single-crossing.
We shall sometimes refer to the separated condition as the multigraph having ``no empty lens,'' where we remark that here a lens is bounded by two entire edges, rather than general edge segments as sometimes defined in the literature.
We also prove a Crossing Lemma variant for separated (and not necessarily locally starlike) multigraphs, where however the term $\alpha \frac{e^3}{n^2}$ must be replaced by $\alpha \frac{e^{2.5}}{n^{1.5}}$.
Both results are best-possible up to the value of constant $\alpha$.
Hence, the Crossing Lemma for separated drawings with $\alpha \frac{e^{2.5}}{n^{1.5}}$ nicely settles between the one for branching drawings with $\alpha \frac{e^3}{n^2}$ (Thm~\ref{thm:CL-branching}) and the one for non-homotopic drawings with $\alpha \frac{e^2}{n}$ (Thm~\ref{thm:CL-non-homotopic}).

\begin{theorem}\label{thm:CL-separated-main}
 There is an absolute constant $\alpha > 0$ such that for any $n$-vertex $e$-edge topological multigraph $G$ with $e > 4n$ we have
 \begin{enumerate}[label = (\roman*)]
  \item ${\rm cr}(G) \geq \alpha \frac{e^3}{n^2}$, if $G$ is separated and locally starlike.\label{enum:CL-sep-nac}
  \item ${\rm cr}(G) \geq \alpha \frac{e^{2.5}}{n^{1.5}}$, if $G$ is separated.\label{enum:CL-separated}
 \end{enumerate}
 Moreover, both bounds are best-possible up to the constant $\alpha$.
\end{theorem}

We prove Theorem~\ref{thm:CL-separated-main} in Section~\ref{sec:separated-styles}.
Our arguments hold in a more general setting, which we present in Section~\ref{sec:general-CL}.
In Section~\ref{sec:new-proofs-old-variants} we use this general setting to deduce other known Crossing Lemma variants, including Theorem~\ref{thm:multiplicity-crossing-lemma}.
We conclude the paper with some open questions in Section~\ref{sec:conclusions}.

\section{A Generalized Crossing Lemma}
\label{sec:general-CL}

In this section we consider general drawing styles and propose a generalized Crossing Lemma, which will subsume the Crossing Lemma variants in Theorem~\ref{thm:CL-separated-main} and Section~\ref{sec:new-proofs-old-variants}.
A \emph{drawing style} $D$ is a predicate over the collection of all topological drawings, i.e., for each topological drawing of a multigraph $G$ we specify whether $G$ is in drawing style $D$ or not.
We say that $G$ is a multigraph in drawing style $D$ when $G$ is a topological multigraph whose drawing is in drawing style $D$.

In order to prove our generalized Crossing Lemma, we follow the line of arguments of Pach and T\'oth~\cite{PT20} for branching multigraphs.
Their main tool is a bisection theorem for branching drawings, which easily generalizes to all separated drawings. 
We generalize their definition as follows.

\begin{definition}[$D$-bisection width]
 For a drawing style $D$ the {\em $D$-bisection width} ${\rm b}_D(G)$ of a multigraph $G$ in drawing style $D$ is the smallest number of edges whose removal splits $G$ into two multigraphs, $G_1$ and $G_2$, in drawing style $D$ with no edge connecting them such that $|V(G_1)|, |V(G_2)|\ge n/5$.
\end{definition} 

We say that a drawing style is \emph{monotone} if removing edges retains the drawing style, that is, for every multigraph $G$ in drawing style $D$ and any edge removal, the resulting multigraph with its inherited drawing from $G$ is again in drawing style $D$.
Note that we require a monotone drawing style to be retained only after removing edges, but not necessarily after removing vertices.
For example, the branching drawing style is in general not maintained after removing a vertex, since a closed curve formed by a pair of parallel edges might become empty.
However, the separated, single-crossing and locally starlike drawings styles (and therefore also the branching drawing style) are monotone.

Given a topological multigraph $G$, we call any operation of the following form a \emph{vertex split}: (1) Replace a vertex $v$ of $G$ by two vertices $v_1$ and $v_2$ and (2) by locally modifying the edges in a small neighborhood of $v$, connect each edge in $G$ incident to $v$ to either $v_1$ or $v_2$ in such a way that no new crossing is created.
Note that such a split is possible, even enforcing the degree of $v_1$ to be any specific number between $0$ and the degree of $v$.
We say that a drawing style is \emph{split-compatible} if performing vertex splits retains the drawing style, that is, for every multigraph $G$ in drawing style $D$ and any vertex split, the resulting multigraph with its inherited drawing from $G$ is again in drawing style $D$.
Again, the separated, single-crossing and locally starlike drawings styles (and therefore also the branching drawing style) are split-compatible.

We are now ready to state our main result.
Recall that $\Delta(G)$ denotes the maximum degree of a vertex in $G$.

\begin{theorem}[Generalized Crossing Lemma]\label{thm:general-drawing-style}
 Suppose $D$ is a monotone and split-compatible drawing style, and that there are constants $k_1,k_2,k_3 > 0$ and $b > 1$ such that each of the following holds for every $n$-vertex $e$-edge multigraph $G$ in drawing style $D$:
 \begin{enumerate}[label = \textbf{(P\arabic*)}, leftmargin = 3em]
  \item If ${\rm cr}(G) = 0$, then the edge count satisfies $e \leq k_1\cdot n$. \label{enum:planar-edge-count}
  
  \item The $D$-bisection width satisfies $b_D(G) \leq k_2 \sqrt{{\rm cr}(G) + \Delta(G)\cdot e + n}$. \label{enum:bisection-width}
  
  \item The edge count satisfies $e \leq k_3 n^b$. \label{enum:general-edge-count}
 \end{enumerate}
 Then there exists an absolute constant $\alpha > 0$ such that for any $n$-vertex $e$-edge multigraph $G$ in drawing style $D$ we have
 \[
  {\rm cr}(G) \geq \alpha \frac{e^{x(b)+2}}{n^{x(b)+1}}, \qquad \text{provided } e > (k_1+1)n,
 \]
 where $x(b) := 1/(b-1)$ and $\alpha$ 
 is some positive constant depending only on $b,k_2$, and $k_3$.
\end{theorem}

\begin{lemma}\label{lem:best-possible}
 If there exist for arbitrarily large $n$ multigraphs in drawing style $D$ with $n$ vertices and $e = \Theta(n^b)$ edges such that any two edges cross at most a constant number of times, then the bound in Theorem~\ref{thm:general-drawing-style} is asymptotically tight.
\end{lemma}
\begin{proof}
 Consider such an $n$-vertex $e$-edge multigraph in drawing style $D$.
 Clearly, there are at most $O(e^2) = O(n^{2b})$ crossings, while Theorem~\ref{thm:general-drawing-style} gives with $x(b)=1/(b-1)$ that there are at least
 \[
  \Omega \left( \frac{e^{x(b)+2}}{n^{x(b)+1}} \right)
  = \Omega \left( \frac{e^{x(b)+2}}{n^{b\cdot x(b)}} \right)
  = \Omega \left( \frac{n^{b\cdot x(b)+2b}}{n^{b\cdot x(b)}} \right)
  = \Omega \left( n^{2b} \right)
 \]
 crossings.
%
%
\end{proof}

\subsection{Proof of Theorem~\ref{thm:general-drawing-style}}

\begin{proofidea}
 Before proving Theorem~\ref{thm:general-drawing-style}, let us sketch the rough idea.
 Suppose, for a contradiction, that $G$ is a multigraph in drawing style $D$ with fewer than $\alpha\frac{e^{x(b)+2}}{n^{x(b)+1}}$ crossings, for a constant $\alpha$ to be defined.
 First, we conclude from~\ref{enum:planar-edge-count} that $G$ must have many edges.
 Then, by~\ref{enum:bisection-width}, the $D$-bisection width of $G$ is small, and thus we can remove few edges from the drawing to obtain two smaller multigraphs, $G_1$ and $G_2$, both also in drawing style $D$, which we call parts.
 We then repeat splitting each large enough part into two parts each, again using~\ref{enum:bisection-width}.
 Note that each part has at most $4/5$ of the vertices of the corresponding part in the previous step.
 We continue until all parts are smaller than a carefully chosen threshold.
 As we removed relatively few edges during this decomposition algorithm, the final parts still have a lot of edges, while having few vertices each.
 This will contradict~\ref{enum:general-edge-count} and hence complete the proof.
\end{proofidea}
 
\noindent
Now, let us start with the proof of Theorem~\ref{thm:general-drawing-style}.
We define an absolute constant
\begin{equation*}
 \alpha := \min\left\{ \frac{1}{2^{2x(b)+16}} \cdot \frac{1}{k_2^2} \cdot \frac{1}{k_3^{x(b)}} \; ; \; \frac{1}{2^{(2x(b)+16)\cdot \frac{x(b)+2}{x(b)}}} \cdot \frac{1}{k_2^{2 \cdot \frac{x(b)+2}{x(b)}}} \cdot \frac{1}{k_3^{x(b)+2}} \right\}.
\end{equation*}
Then a simple computation shows that
\begin{align}
 \sqrt{\alpha} \cdot k_2 \cdot \sqrt{k_3^{x(b)}} \cdot 2^{x(b)+6} &\leq \frac{1}{4} \text{ and }\label{eq:alpha-bound-1}\\
 \sqrt{\alpha^{\frac{x(b)}{x(b)+2}}} \cdot k_2 \cdot \sqrt{k_3^{x(b)}} \cdot 2^{x(b)+6} &\leq \frac{1}{4},\label{eq:alpha-bound-2}
\end{align}
which will be important later.

Now let $\tilde{G}$ be a fixed multigraph in drawing style $D$ with $\tilde{n}$ vertices and $\tilde{e} > (k_1+1)\tilde{n}$ edges.
Let $G'$ be an edge-maximal subgraph of $\tilde{G}$ on vertex set $V(\tilde{G})$ such that the inherited drawing of $G'$ has no crossings.
Since $D$ is monotone, $G'$ is in drawing style $D$.
Hence, by~\ref{enum:planar-edge-count}, for the number $e'$ of edges in $G'$ we have $e' \leq k_1\cdot n' = k_1 \cdot \tilde{n}$.
Since $G'$ is edge-maximal crossing-free, each edge in $E(\tilde{G}) - E(G')$ has at least one crossing with an edge in $E(G')$.
Thus
\begin{equation}\label{eq:linear-crossing-number}
 {\rm cr}(\tilde{G}) \geq \tilde{e}-e' \geq \tilde{e} - k_1\tilde{n} > \tilde{n}.
\end{equation}
In case $(k_1+1)\tilde{n} < \tilde{e} \leq \beta\tilde{n}$ for $\beta := \alpha^{-1/(x(b)+2)}$, we get
\[
 {\rm cr}(\tilde{G}) \overset{\eqref{eq:linear-crossing-number}}{>} \tilde{n}
 \geq \alpha \cdot \frac{\tilde{e}^{x(b)+2}}{\tilde{n}^{x(b)+1}},
\]
as desired.
To prove Theorem~\ref{thm:general-drawing-style} in the remaining case $\tilde{e} > \beta\tilde{n}$ we use proof by contradiction. 
Therefore assume that the number of crossings in $\tilde{G}$ satisfies
\[
 {\rm cr}(\tilde{G}) < \alpha \cdot \frac{ \tilde{e}^{x(b)+2}}{\tilde{n}^{x(b)+1}}.
\]
Let $d$ denote the average degree of the vertices of $\tilde{G}$, that is, $d=2\tilde{e}/\tilde{n}$.
For every vertex $v\in V(\tilde{G})$ whose degree, $\deg(v,\tilde{G})$, is larger than $d$, we perform $\lceil \deg(v,\tilde{G})/d \rceil - 1$ vertex splits so as to split $v$ into $\lceil \deg(v,\tilde{G})/d \rceil$ vertices, each of degree at most~$d$.
At the end of the procedure, we obtain a multigraph $G$ with $e=\tilde{e}$ edges, $n<2\tilde{n}$ vertices, and maximum degree $\Delta(G) \leq d = 2\tilde{e}/\tilde{n} < 4e/n$.
Moreover, as $D$ is split-compatible, $G$ is in drawing style $D$.
For the number of crossings in $G$, we have
\begin{equation}\label{start}
 {\rm cr}(G) ={\rm cr}(\tilde{G}) < \alpha \cdot \frac{ \tilde{e}^{x(b)+2}}{\tilde{n}^{x(b)+1}} < 2^{x(b)+1} \alpha \cdot \frac{e^{x(b)+2}}{n^{x(b)+1}}.
\end{equation}
Moreover, recall that
\begin{equation}\label{eq:delta}
 e > \beta\tilde{n} > \beta \frac{n}{2} \qquad \text{for } \beta = \frac{1}{\alpha^{1/(x(b)+2)}}.
\end{equation}
We break $G$ into smaller parts, according to the following procedure.
At each step the parts form a partition of the entire vertex set $V(G)$.

\begin{center}
\begin{minipage}{0.8\textwidth}
 \noindent {\sc Decomposition Algorithm}

 \medskip

 \noindent {\sc Step 0.}\\
 $\triangleright$ {\bf Let} $G^0=G, G^0_1=G, M_0=1, m_0=1.$

 \medskip

 Suppose that we have already executed {\sc Step} $i$, and that the resulting graph $G^{i}$ consists of $M_i$ parts, $G_1^{i},G_2^{i},\ldots,G_{M_{i}}^{i}$, each in drawing style $D$ and having at most~$(4/5)^i n$ vertices.
 Assume without loss of generality that each of the first $m_i$ parts of $G^i$ has at least $(4/5)^{i+1}n$ vertices and the remaining $M_i-m_i$ have fewer.
 Letting $n(G_j^i)$ denote the number of vertices of the part $G_j^i$, we have
 \begin{equation*}
  (4/5)^{i+1} n(G) \le n(G_j^{i}) \le (4/5)^{i} n(G), \qquad 1\le j\le {m_{i}}.
 \end{equation*}
 Hence,
 \begin{equation}\label{darab}
  m_i \le (5/4)^{i+1}.
 \end{equation}

 \smallskip

 \noindent {\sc Step $i+1$.}\\
 $\triangleright$ {\bf If}
 \begin{equation*}
  (4/5)^{i} < \frac{1}{(2k_3)^{x(b)}} \cdot \frac{e^{x(b)}}{n^{x(b)+1}},
 \end{equation*}
 {\bf then} {\sc stop}.\\
 $\triangleright$ {\bf Else}, for $j=1,2,\ldots,{m_{i}}$, delete ${\rm b}_{\rm D}(G_j^{i})$ edges from $G_j^{i}$, as guaranteed by~\ref{enum:bisection-width}, such that $G_j^{i}$ falls into two parts, each of which is in drawing style $D$ and contains at most $(4/5)n(G_j^{i})$ vertices.
 Let $G^{i+1}$ denote the resulting graph on the original set of $n$ vertices.
 
 \medskip
 
 Clearly, each part of $G^{i+1}$ has at most $(4/5)^{i+1}n$ vertices.
\end{minipage}
\end{center}

\noindent
Suppose that the {\sc Decomposition Algorithm} terminates in {\sc Step} $k+1$.
If $k>0$, then
\begin{equation}\label{kettes}
 (4/5)^k < \frac{1}{(2k_3)^{x(b)}} \cdot \frac{e^{x(b)}}{n^{x(b)+1}} \le (4/5)^{k-1}.
\end{equation}

First, we give an upper bound on the total number of edges deleted from $G$.
Using Cauchy-Schwarz inequality, we get for any nonnegative numbers $a_1,\ldots,a_m$,
\begin{equation}\label{harmas}
 \sum_{j=1}^m\sqrt{a_j}\le \sqrt{m\sum_{j=1}^m a_j},
\end{equation}
and thus obtain that, for any $0 \le i \le k$,
\begin{multline}\label{eq:cr-upper-bound}
 \sum_{j=1}^{m_i} \sqrt{{\rm cr}(G_j^{i})}
 \overset{\eqref{harmas}}{\le} \sqrt{m_i \sum_{j=1}^{m_i} {\rm cr}(G_j^i)}
 \overset{\eqref{darab}}{\le} \sqrt{(5/4)^{i+1}} \sqrt{{\rm cr}(G)}\\
 \overset{\eqref{start}}{<} \sqrt{(5/4)^{i+1}} \sqrt{2^{x(b)+1} \alpha \cdot \frac{e^{x(b)+2}}{n^{x(b)+1}}}.
\end{multline}
Letting $e(G_j^i)$ and $\Delta(G^i_j)$ denote the number of edges and maximum degree in part $G^i_j$, respectively, we obtain similarly
\begin{multline}\label{eq:deg-2-upper-bound}
 \sum_{j=1}^{m_i} \sqrt{ \Delta(G^i_j) \cdot e(G_j^i) + n(G^i_j)}
 \overset{\eqref{harmas}}{\le} \sqrt{ m_i \left( \sum_{j=1}^{m_i} \Delta(G^i_j) \cdot e(G_j^i) + n(G_j^i) \right)} \\
 \overset{\eqref{darab}}{\le} \sqrt{(5/4)^{i+1}} \sqrt{ \Delta(G) \cdot e + n} 
 \le \sqrt{(5/4)^{i+1}} \sqrt{ \frac{4e}{n}e + n} \\
 < \sqrt{(5/4)^{i+1}} \sqrt{\frac{4e^2}{n}+\frac{4e^2}{n}}
 < \sqrt{(5/4)^{i+1}} \frac{3e}{\sqrt{n}},
\end{multline}
where we used in the last line the fact that $n/2 < e$.

Using a partial sum of a geometric series we get
\begin{equation}\label{eq:geometric-series}
 \sum_{i=0}^{k} (\sqrt{5/4})^{i+1} 
 = \frac{(\sqrt{5/4})^{k+2} - 1}{\sqrt{5/4}-1} - 1
 < \frac{(\sqrt{5/4})^3}{\sqrt{5/4}-1} \cdot (\sqrt{5/4})^{k-1}
 < 12 \cdot (\sqrt{5/4})^{k-1}
\end{equation}
Thus, as each $G_j^i$ is in drawing style $D$ and hence~\ref{enum:bisection-width} holds for each $G_j^i$, the total number of edges deleted during the decomposition procedure is
\begin{align}
 \sum_{i=0}^{k} \sum_{j=1}^{m_{i}} {\rm b}_D(G_j^{i})
 \le{}& k_2 \sum_{i=0}^{k} \sum_{j=1}^{m_{i}} \sqrt{{\rm cr}(G_j^{i}) + \Delta(G_j^i) \cdot e(G_j^i) + n(G_j^i)} \nonumber \\
 \le{}& k_2 \left( \sum_{i=0}^{k} \sum_{j=1}^{m_{i}} \sqrt{{\rm cr}(G_j^{i})} + \sum_{i=0}^{k} \sum_{j=1}^{m_{i}} \sqrt{ \Delta(G_j^i) \cdot e(G_j^i) + n(G_j^i)} \right) \nonumber \\
 \overset{\eqref{eq:cr-upper-bound},\eqref{eq:deg-2-upper-bound}}{\le}{}& k_2 \left( \sum_{i=0}^{k} \sqrt{(5/4)^{i+1}} \right) \left( \sqrt{2^{x(b)+1}\alpha \cdot \frac{e^{x(b)+2}}{n^{x(b)+1}}} + \frac{3e}{\sqrt{n}} \right) \nonumber \\
 \overset{\eqref{eq:geometric-series}}{<}{}& k_2 \cdot 12 \sqrt{(5/4)^{k-1}} \left( \sqrt{2^{x(b)+1}\alpha \cdot \frac{e^{x(b)+2}}{n^{x(b)+1}}} + \frac{3e}{\sqrt{n}} \right) \nonumber \\
 \overset{\eqref{kettes}}{<}{}& k_2 \cdot 12 \sqrt{ (2k_3)^{x(b)} \cdot \frac{n^{x(b)+1}}{e^{x(b)}} } \left( \sqrt{2^{x(b)+1}\alpha \cdot \frac{e^{x(b)+2}}{n^{x(b)+1}}} + \frac{3e}{\sqrt{n}} \right) \nonumber \\
 <{}& k_2 \cdot 36 \cdot \sqrt{k_3^{x(b)}} \left( 2^{x(b)} \sqrt{\alpha} e + \sqrt{ \frac{2^{x(b)}n^{x(b)}}{e^{x(b)-2}} } \right) \nonumber \\
 \overset{\eqref{eq:delta}}{<}{}& k_2 \cdot 36 \cdot \sqrt{k_3^{x(b)}} \cdot 2^{x(b)} \left( \sqrt{\alpha} + \sqrt{ \frac{1}{\beta^{x(b)}} } \right) e \nonumber \\
 \overset{\eqref{eq:delta}}{=}{}& k_2 \cdot 36 \cdot \sqrt{k_3^{x(b)}} \cdot 2^{x(b)} \left( \sqrt{\alpha} + \sqrt{ \alpha^{\frac{x(b)}{x(b)+2}} } \right) e \nonumber \\
 <{}& k_2 \cdot \sqrt{k_3^{x(b)}} \cdot 2^{x(b)+6} \left( \sqrt{\alpha} + \sqrt{ \alpha^{\frac{x(b)}{x(b)+2}} } \right) e \overset{\eqref{eq:alpha-bound-1},\eqref{eq:alpha-bound-2}}{\leq} \frac{e}{2}. \label{eq:removed-edges-upper-bound}
\end{align}
By~\eqref{eq:removed-edges-upper-bound} the {\sc Decomposition Algorithm} removes less than half of the edges of $G$ if $k > 0$.
Hence, the number of edges of the graph $G^k$ obtained in the final step of this procedure satisfies
\begin{equation}\label{vegso}
 e(G^k) > \frac{e}{2}.
\end{equation}
(Note that this inequality trivially holds if the algorithm terminates in the very first step, i.e., when $k=0$.)

\smallskip

Next we shall give an upper bound on $e(G^k)$ that contradicts~\eqref{vegso}.
The number of vertices of each part $G_j^k$ of $G^k$ satisfies
\[
 n(G_j^{k}) \le (4/5)^k n
 \overset{\eqref{kettes}}{<} \left( \frac{1}{(2k_3)^{x(b)}} \cdot \frac{e^{x(b)}}{n^{x(b)+1}} \right) n
 = \left( \frac{e}{2 \cdot k_3 \cdot n} \right)^{x(b)},\quad 1 \le j \le M_k.
\]
Hence
\[
 n(G_j^k)^{b-1}
 < \left( \frac{e}{2 \cdot k_3 \cdot n} \right)^{x(b)(b-1)} 
 = \frac{e}{2 \cdot k_3 \cdot n},
\]
since $x(b) = 1/(b-1)$ and hence $x(b)(b-1) = 1$.

As $G_j^k$ is in drawing style $D$,~\ref{enum:general-edge-count} holds for $G_j^k$ and we have
\[
 e(G_j^{k})
 \le k_3 \cdot n(G_j^{k})^{b} 
 < k_3 \cdot n(G_j^{k}) \cdot \frac{e}{2 \cdot k_3 \cdot n}
 = n(G_j^{k}) \cdot \frac{e}{2n}.
\]
Therefore, for the total number of edges of $G^k$ we have
\[
 e(G^k) = \sum_{j=1}^{M_k}e(G_j^{k}) < \frac{e}{2n}\sum_{j=1}^{M_k}n(G_j^{k}) = \frac{e}{2},
\]
contradicting~\eqref{vegso}.
This completes the proof of Theorem~\ref{thm:general-drawing-style}.\qed

\section{Separated Multigraphs}
\label{sec:separated-styles}

We derive our Crossing Lemma variants for separated multigraphs (Theorem~\ref{thm:CL-separated-main}) from the generalized Crossing Lemma (Theorem~\ref{thm:general-drawing-style}) presented in Section~\ref{sec:general-CL}.
Let us denote the separated drawing style by $\Dsep$
 and the separated and locally starlike drawing style by $\Dlocstar$.
In order to apply Theorem~\ref{thm:general-drawing-style}, we shall find for $D = \Dsep, \Dlocstar$
{\bfseries (1)} the largest number of edges in a crossing-free $n$-vertex multigraph in drawing style $D$,
{\bfseries (2)} an upper bound on the $D$-bisection width of multigraphs in drawing style $D$, and
{\bfseries (3)} an upper bound on the number of edges in any $n$-vertex multigraph in drawing style $D$.

As for crossing-free multigraphs $\Dsep$ and $\Dlocstar$ are equivalent to the branching drawing style, we can rely on the following Lemma of Pach and T\'oth.

\begin{lemma}[Pach and T\'oth~\cite{PT20}]\label{lem:crossing-free-branching}
 Any $n$-vertex crossing-free branching multigraph, $n \geq 3$, has at most $3n-6$ edges.
\end{lemma}


\begin{corollary}\label{cor:crossing-free-separated}
 Any $n$-vertex crossing-free multigraph in drawing style $\Dsep$ or $\Dlocstar$, $n \geq 3$, has at most $3n-6$ edges.
\end{corollary}
%

Also we can derive the bounds on the $D$-bisection width from the corresponding bound for the branching drawing style due to Pach and T\'oth.

\begin{lemma}[Pach and T\'oth~\cite{PT20}]\label{lem:branching-bisection}
 For any multigraph $G$ in the branching drawing style $D$ with $n$ vertices of degrees $d_1,d_2,\ldots,d_n$, and with ${\rm cr}(G)$ crossings, the $D$-bisection width of $G$ satisfies
 \[
  {\rm b}_D(G) \le 22\sqrt{ {\rm cr}(G) + \sum_{i=1}^{n} d_i^2 + n }.
 \]
\end{lemma}

\begin{lemma}\label{lem:separated-bisection}
 For $D = \Dsep, \Dlocstar$ any multigraph $G$ in the drawing style $D$ with $n$ vertices, $e$ edges, maximum degree $\Delta(G)$, and with ${\rm cr}(G)$ crossings, the $D$-bisection width of $G$ satisfies
 \[
  {\rm b}_D(G) \le 44\sqrt{ {\rm cr}(G) + \Delta(G) \cdot e + n }.
 \]
\end{lemma}
\begin{proof}
 Let $G$ be a multigraph in drawing style $D$.
 Our goal is that introducing a new vertex at each crossing, the resulting crossing-free multigraph is separated.
 As this may fail in general, we might have to redraw $G$ first.
 
 To begin, we remove all selfcrossings of edges by simply rerouting each such edges in a crossing-free way within its original curve.
 Observe that this preserves the drawing style $D$.
 In fact, for $D = \Dsep$, no self-crossing edge has a parallel edge, and thus any pair of parallel edges remains unaltered.
 Since the number of crossings is reduced, we may assume without loss of generality that $G$ has no selfcrossings.
 
 Now suppose there is a simple closed curve $\gamma$ formed by parts of only two edges $e_1$ and $e_2$, which does not have a vertex in its interior.
 This can happen between two crossings of $e_1$ and $e_2$, or for $D \neq \Dlocstar$ between a common endpoint and a crossing of $e_1$ and $e_2$.
 Further assume that the interior of $\gamma$ is inclusion-minimal among all such curves, and note that this implies that an edge crosses $e_1$ along $\gamma$ if and only if it crosses $e_2$ along $\gamma$.
 Say $e_1$ has at most as many crossings along $\gamma$ as $e_2$.
 We then reroute the part of $e_2$ on $\gamma$ very closely along the part of $e_1$ along $\gamma$ so as to reduce the number of crossings between $e_1$ and $e_2$.
 The rerouting does not introduce new crossing pairs of edges.
 Hence, the resulting multigraph is again in drawing style $D$ and has at most as many crossings as $G$.
 Similarly, we proceed when $\gamma$ has no vertex in its exterior.
 
 Thus, we can redraw $G$ to obtain a multigraph $G'$ in drawing style $D$ with ${\rm cr}(G') \leq {\rm cr}(G)$, such that introducing a new vertex at each crossing of $G'$ creates a crossing-free multigraph that is separated.
 Moreover, if $G$ is locally starlike, then so is $G'$.
 I.e., $G'$ is in drawing style $D$ and additionally separated.
 Now, using precisely the same proof as in \cite{PT20} (for Lemma~\ref{lem:branching-bisection}), we can show that
 \[
  {\rm b}_D(G') \le 22\sqrt{ {\rm cr}(G') + \sum_{i=1}^{n} d_i^2 + n },
 \]
 where $d_1,\ldots,d_n$ denote the degrees of vertices in $G'$.
 Thus with 
 \[
  \sum_{i=1}^{n} d_i^2 \leq \Delta(G) \sum_{i=1}^n d_i \leq 2 \Delta(G) \cdot e
 \]
 the result follows.
\end{proof}

Finally, let us bound the number of edges in general (not necessarily crossing-free) multigraphs.
Again, we can utilize the result of Pach and T\'oth for the branching drawing style.

\begin{lemma}[Pach and T\'oth~\cite{PT20}]\label{lem:edge-count-branching}
 For any $n$-vertex $e$-edge, $n \geq 3$, multigraph of maximum degree $\Delta(G)$ in the branching drawing style we have $\Delta(G) \leq 2n-4$ and $e \leq n(n-2)$, and both bounds are best-possible.
\end{lemma}

\begin{lemma}\label{lem:edge-count-separated}
 For any $n$-vertex $e$-edge, $n \geq 3$, multigraph $G$ in drawing style $D$ of maximum degree $\Delta(G)$ we have
 \begin{enumerate}[label = (\roman*)]
  \item $\Delta(G) \leq (n-1)(n-2)$ and $e \leq \binom{n}{2}(n-2)$ if $D = \Dsep$,\label{enum:edge-count-Dsep}
  
  
  \item $\Delta(G) \leq 2n-4$ and $e \leq n(n-2)$ if $G$ if $D = \Dlocstar$.\label{enum:edge-count-Dnac}
 \end{enumerate}
 Moreover, each bound is best-possible.
\end{lemma}
\begin{proof}
 Let $G$ be a fixed $n$-vertex, $n \geq 3$, $e$-edge crossing-free multigraph in drawing style $D$.
 
 \begin{enumerate}[label = (\roman*)]
  \item Let $D = \Dsep$.
   Clearly, every set of pairwise parallel edges contains at most $n-2$ edges, since every lens has to contain a vertex different from the two endpoints of these edges.
   This gives $\Delta(G) \leq (n-1)(n-2)$ and $e \leq n\Delta(G)/2 = \binom{n}{2}(n-2)$.
   To see that these bounds are tight, consider $n$ points in the plane with no four points on a circle.
   Then it is easy to draw between any two points $n-2$ edges as circular arcs such that the resulting multigraph (which has $\binom{n}{2}(n-2)$ edges) is in separating drawing style.
  
  
  \item Let $D = \Dlocstar$.
   Consider any fixed vertex $v$ in $G$ and remove all edges not incident to $v$.
   The resulting multigraph is branching and hence by Lemma~\ref{lem:edge-count-branching} $v$ has at most $2n-4$ incident edges.
   Thus $\Delta(G) \leq 2n-4$ and $e \leq n\Delta(G)/2 = n(n-2)$.
   By Lemma~\ref{lem:edge-count-branching}, these bounds are tight, even for the more restrictive branching drawing style.
 \end{enumerate}
\end{proof}


We are now ready to prove that drawing styles $\Dlocstar$ and $\Dsep$ fulfill the requirements of the generalized Crossing Lemma (Theorem~\ref{thm:general-drawing-style}), which lets us prove Theorem~\ref{thm:CL-separated-main}.

\begin{proof}[Proof of Theorem~\ref{thm:CL-separated-main}]
 Let $D = \Dlocstar$ for~\ref{enum:CL-sep-nac} and $D = \Dsep$ for~\ref{enum:CL-separated}.
 Clearly, these drawing styles are monotone, i.e., maintained when removing edges, as well as split-compatible.
 So it remains to determine the constants $k_1,k_2,k_3 > 0$ and $b > 1$ such that~\ref{enum:planar-edge-count}, \ref{enum:bisection-width}, and~\ref{enum:general-edge-count} hold for $D$.
 
 \ref{enum:planar-edge-count} holds with $k_1 = 3$ for $D = \Dlocstar, \Dsep$ by Corollary~\ref{cor:crossing-free-separated}.
 \ref{enum:bisection-width} holds with $k_2 = 44$ for $D = \Dlocstar, \Dsep$ by Lemma~\ref{lem:separated-bisection}.
 \ref{enum:general-edge-count} holds with $k_3 = 1$ and $b = 3$ for $D = \Dsep$ by Lemma~\ref{lem:edge-count-separated}\ref{enum:edge-count-Dsep}, and with $k_3 = 1$ and $b = 2$ for $D = \Dlocstar$ by Lemma~\ref{lem:edge-count-separated}\ref{enum:edge-count-Dnac}.
 
 For $b = 2$ we have $x(b) = 1/(b-1) = 1$.
 Thus Theorem~\ref{thm:general-drawing-style} for $D = \Dlocstar$ gives an absolute constant $\alpha > 0$ such that for every $n$-vertex $e$-edge separated and locally starlike multigraph we have ${\rm cr}(G) \geq \alpha e^{x(b)+2}/n^{x(b)+1} = \alpha e^3/n^2$, provided $e > (k_1+1)n = 4n$.
 Moreover, by Lemma~\ref{lem:edge-count-separated}\ref{enum:edge-count-Dnac} there are separated multigraphs with $n$ vertices and $\Theta(n^2)$ edges, any two of which cross at most once.
 Hence, the term $e^3/n^2$ is best-possible by Lemma~\ref{lem:best-possible}.

 For $b = 3$ we have $x(b) = 1/(b-1) = 0.5$.
 Thus Theorem~\ref{thm:general-drawing-style} for $D = \Dsep$ gives an absolute constant $\alpha > 0$ such that for every $n$-vertex $e$-edge separated multigraph we have ${\rm cr}(G) \geq \alpha e^{x(b)+2}/n^{x(b)+1} = \alpha e^{2.5}/n^{1.5}$, provided $e > (k_1+1)n = 4n$.
 Moreover, by Lemma~\ref{lem:edge-count-separated}\ref{enum:edge-count-Dsep} there are separated multigraphs with $n$ vertices and $\Theta(n^3)$ edges, any two of which cross at most twice.
 Hence, the term $e^{2.5}/n^{1.5}$ is best-possible by Lemma~\ref{lem:best-possible}. 
\end{proof}

\section{Other Crossing Lemma Variants}
\label{sec:new-proofs-old-variants}

We use the generalized Crossing Lemma (Theorem~\ref{thm:general-drawing-style}) to reprove existing variants of the Crossing Lemma due to Sz\'ekely~\cite{Sz97} and Pach, Spencer, and T\'oth~\cite{PST00}, respectively.

\subsection{Low Multiplicity}

Here we consider for fixed $m \geq 1$ the drawing style $D_m$ which is characterized by the absence of $m+1$ pairwise parallel edges.
In particular, any $n$-vertex multigraph $G$ in drawing style $D_m$ has at most $m\binom{n}{2}$ edges, i.e., \ref{enum:general-edge-count} holds for $D_m$ with $b = 2$ and $k_3 = m$.
Moreover, if $G$ is crossing-free on $n$ vertices and $e$ edges, then $e \leq 3mn$, i.e., \ref{enum:planar-edge-count} holds for $D_m$ with $k_1 = 3m$.

Finally, we claim that \ref{enum:bisection-width} holds for $D_m$ with $k_2$ being independent of $m$.
To this end, let $G$ be any $n$-vertex $e$-edge multigraph in drawing style $D_m$.
As already noted by Sz\'ekely~\cite{Sz97}, we can reroute all but one edge in each bundle in such a way that in the resulting multigraph $G'$ every lens is empty, no two adjacent edges cross, and ${\rm cr}(G') \leq {\rm cr}(G)$.
(Simply route every edge very closely to its parallel copy with the fewest crossings.)
Clearly, $G'$ has drawing style $D_m$.

Now, we place a new vertex in each lens of $G'$, giving a multigraph $G''$ with $n'' \leq n+e$ vertices and $e'' = e$ edges, which is in the separated drawing style $D$.
By Lemma~\ref{lem:separated-bisection}, there is an absolute constant $k$ such that
\[
 b_D(G'') \leq k \sqrt{ {\rm cr}(G'') + \Delta(G'') \cdot e'' + n'' }.
\]
As $b_{D_m}(G) \leq b_D(G'')$, ${\rm cr}(G'') = {\rm cr}(G') \leq {\rm cr}(G)$, $\Delta(G'') = \Delta(G)$, and $\Delta(G) + 1 \leq 2\Delta(G)$ we conclude that
\[
 b_{D_m}(G) \leq 2k \sqrt{ {\rm cr}(G) + \Delta(G) \cdot e + n}.
\]
In other words, \ref{enum:bisection-width} holds for drawing style $D_m$ with an absolute constant $k_2 = 2k$ that is independent of~$m$.

Note that for $b = 2$, we have $x(b) = 1$.
We conclude with Theorem~\ref{thm:general-drawing-style} that there is an absolute constant $\alpha'$ such that for every $m$ and every $n$-vertex $e$-edge multigraph $G$ in drawing style $D_m$ we have
\[
 {\rm cr}(G) \geq \alpha' \cdot \frac{1}{k_3^{x(b)}} \cdot \frac{e^{x(b)+2}}{n^{x(b)+1}} = \alpha' \cdot \frac{e^3}{mn^2}, \qquad \text{provided } e > (3m+1)n,
\]
which is the statement of Theorem~\ref{thm:multiplicity-crossing-lemma}; except that we slightly improved the assumption of $e > 5mn$ in Theorem~\ref{thm:multiplicity-crossing-lemma} to $e > (3m+1)n$.

\subsection{High Girth}

\begin{oldtheorem}[Pach, Spencer, T\'oth~\cite{PST00}]\label{thm:girth-crossing-lemma}
 For any $r \geq 1$ there is an absolute constant $\alpha_r > 0$ such that for any $n$-vertex $e$-edge graph $G$ of girth larger than $2r$ we have
 \[
  {\rm cr}(G) \geq \alpha_r \cdot \frac{e^{r+2}}{n^{r+1}}, \qquad \text{provided } e > 4n.
 \]
\end{oldtheorem}

Here we consider for fixed $r \geq 1$ the drawing style $D_r$ which is characterized by the absence of cycles of length at most $2r$.
In particular, any multigraph $G$ in drawing style $D_r$ has neither loops nor multiple edges.
Hence \ref{enum:planar-edge-count} holds for drawing style $D_r$ with $k_1 = 3$.
Secondly, drawing style $D_r$ is more restrictive than the separated drawing style and thus also \ref{enum:bisection-width} holds for $D_r$.
Moreover, any $n$-vertex graph in drawing style $D_r$ has $O(n^{1+1/r})$ edges~\cite{AHL02}, i.e., \ref{enum:general-edge-count} holds for $D_r$ with $b = 1+1/r$.
Finally, $D_r$ is obviously a monotone and split-compatible drawing style.

Thus with $x(b) = 1/(b-1) = r$, Theorem~\ref{thm:general-drawing-style} immediately gives the existence of an absolute constant $\alpha_r$ such that
\[
 {\rm cr}(G) \geq \alpha_r \cdot \frac{e^{r+2}}{n^{r+1}}, \qquad \text{provided } e > 4n
\]
for any $n$-vertex $e$-edge multigraph in drawing style $D_r$, which is the statement of Theorem~\ref{thm:girth-crossing-lemma}.

\section{Conclusions}
\label{sec:conclusions}

Let $G$ be a topological multigraph with $n$ vertices and $e > 4n$ edges.
We have shown that ${\rm cr}(G) \geq \alpha e^3/n^2$ if $G$ is separated and locally starlike, which generalizes the result for branching multigraphs~\cite{PT20}, which are additionally single-crossing.
Moreover, if $G$ is only separated, then the lower bound drops to ${\rm cr}(G) \geq \alpha e^{2.5}/n^{1.5}$, which is tight up to the constant factor, too.
It remains open to determine a best-possible Crossing Lemma for separated and single-crossing multigraphs.
This would follow from our generalized Crossing Lemma (Theorem~\ref{thm:general-drawing-style}), where the missing ingredient is the determination of the smallest $b$ such that every separated and single-crossing multigraph $G$ on $n$ vertices has $O(n^b)$ edges.
It is easy to see that the maximum degree $\Delta(G)$ may be as high as $(n-1)(n-2)$, but we suspect that any such $G$ has $O(n^2)$ edges.
This has been recently verified up to a logarithmic factor, see~\cite{FPS21}.

\section*{Acknowledgements}
This project initiated at the Dagstuhl seminar 16452 ``Beyond-Planar Graphs: Algorithmics and Combinatorics,'' November 2016.
We would like to thank all participants, especially Stefan Felsner, Vincenzo Roselli, and Pavel Valtr, for fruitful discussions.

\bibliographystyle{abbrv}
\bibliography{lit}

\end{document}